\documentclass[reqno, 12pt]{amsart}

\pdfoutput=1

\usepackage{latexsym}
\usepackage[centertags]{amsmath}
\usepackage{amsfonts}
\usepackage{amssymb}
\usepackage{amsthm}
\usepackage{newlfont}
\usepackage{graphics}
\usepackage{color}
\usepackage{float}
\usepackage{diagbox}
\usepackage{longtable}
\usepackage{rotating}
\usepackage{multirow}

\usepackage{extarrows}

\usepackage[sort,compress,numbers]{natbib}

\usepackage[utf8]{inputenc}

\newtheorem{thm}{Theorem}

\newtheorem{prop}[thm]{Proposition}

\newtheorem{conj}[thm]{Conjecture}

\theoremstyle{mydefinition}

\theoremstyle{myremark}

\allowdisplaybreaks[1]

\def\pa[1]{\frac{\partial}{\partial x}}

\newcommand{\qfac}[1]{(q)_n}

\title{Hankel determinants for convolution powers of Motzkin numbers}

\author{Ying Wang$^1$ and Yingrui Zhang$^{2,*}$}

\address{ $^{1}$School of Mathematics and Statistics, North China University of Water Resources and Electric Power,
 Zhengzhou 450045, PR China
 }

 \address{$^{2}$School of Statistics and Mathematics, Yunnan University of Finance and Economics, Kunming 650221, PR China
 }

\email{$^1$\texttt{wangying2019@ncwu.edu.cn}\ \ \ \& $^2$\texttt{zyrzuhe@126.com}}
\date{\today}

\begin{document}

\maketitle

\begin{abstract}
 We
evaluate the Hankel determinants of the convolution powers of Motzkin numbers for $r\leq 27$ by finding shifted periodic continued fractions, which
arose in application of Sulanke and Xin's continued fraction method. We also conjecture some polynomial characterization of these determinants.
\end{abstract}

\noindent
\begin{small}
 \emph{Mathematic subject classification}: Primary 15A15; Secondary 05A15, 11B83.
\end{small}

\noindent
\begin{small}
\emph{Keywords}: Hankel determinants; Continued fractions; Motzkin numbers.
\end{small}

\newcommand\con[1]{\equiv_{#1}}

\section{Introduction}
This paper is along the line of using generating functions to deal with Hankel determinants.
The Hankel determinant of a generating function $A(x)=\sum_{n\geq 0}a_nx^n$ is defined by
 $$ H_n(A(x)) = \det ( a_{i+j} )_{0\le i,j\le n-1}, \qquad H_0(A(x))=1.$$

In recent years, a considerable amount of work has been devoted to Hankel determinant evaluations of various counting coefficients.
Many of such Hankel determinants have attractive compact closed formulas,  such as
binomial coefficients, Catalan numbers \cite{J.M.E.Mays J.Wojciechowski}, Motzkin numbers \cite{M. Aigner,J.Cigler}, large and little Schr\"oder numbers \cite{R. A Brualdi and S. Kirkland.}.
See
 \cite{M. Aigner Catalan-like, M. Aigner Catalan, P. Barry,P. Barry Catalan, D.M. Bressoud, R. A Brualdi and S. Kirkland., M. Elouafi, I. Gessel and G. Viennot., Q.-H. Hou-A. Lascoux-Y.-P. Mu,C. Krattenthaler.1999, C. Krattenthaler.,C. Krattenthaler.2010,MuLili-WangYi-YeNan,MuLili-WangYi, Sulanke-Xin,Tamm} for further references.

The classical method of continued fractions, either by $J$-fractions \cite{C. Krattenthaler.,H. S. Wall},
or by $S$-fractions \cite{W. B. Jones and W. J. Thron}, requires $H_n(A(x))\neq 0$ for all $n$.
There also have Gessel-Xin's \cite{Gessel and Xin}, Han's\cite{GuoNiuHan} and Sulanke-Xin's \cite{Sulanke-Xin} continued fraction method allow $H_n(A(x))=0$ for some values of $n$. Sulanke-Xin defined a quadratic transformation $\tau$ such that $H(F(x))$ and $H(\tau(F(x)))$ have simple connections.
The shifted periodic continued fractions (of order $q$) of the form
 $$F_0^{(p)}(x) \mathop{\longrightarrow}\limits^\tau F_1^{(p)}(x)\mathop{\longrightarrow}\limits^\tau \cdots \mathop{\longrightarrow}\limits^\tau F_q^{(p)}(x)=F_0^{(p+1)}(x) $$
were found in \cite{Y. Wang and G. Xin}, which appear in Hankel determinants of many path counting numbers. Here $p$ is an additional parameter. If one can guess an explicit formula for $F_0^{(p)}(x)$ by Maple,
then their Hankel determinants can be easily computed. 
This is also the primary method of this paper.

We have know that the Hankel determinants for the convolution power of Catalan numbers exhibits some perfect results and Conjectures \cite{Y. Wang-G. Xin2}.
As we also know the close relationship between Motzkin numbers and Catalan numbers,
it is necessary to investigate the Hankel determinants for the convolution power of Motzkin numbers.
$$H_n(M(x)^r)= \det(M_n^{(r)})_{0\le i,j\le n-1}$$
is the $r$-th convolution power of the well-known Motzkin numbers $M_n=\displaystyle \sum_{k\geq 0}\binom{n}{2k}C_k$. $C_k$ is the Catalan numbers.
The name comes after the
generating function identity
\begin{align}
  \sum_{n\ge 0} M_n^{(r)} x^n = M(x)^r = \Big(\sum_{n\ge 0} M_n x^n \Big)^r.
\end{align}

In what follows, we will denote by $F(x,r)=M(x)^r$.
The cases $r=1$ are well-known \cite{Cigler-Krattenthaler-Motzkin}.
\begin{align*}
 H_n(F(x,1))&= H_n(M(x))=  1.
\end{align*}

The cases $r=2,3$ are nice.
\begin{thm}\label{example-F3}
\begin{align*}
  H_n(F(x,2))&=\left\{
\begin{aligned}
 1, &\text{\ \  n mod 12}=0,1,2,3.&\\
 0,& \text{\ \  n mod 12}=4,5,10,11.&\\
 -1,& \text{\ \  n mod 12}=6,7,8,9.&
\end{aligned}
\right.
\end{align*}

\begin{align*}
 H_{3n}(F(x,3))=H_{3n+1}(F(x,3))=(-1)^n(n+1)^2,\  H_{3n+2}(F(x,3))=0.
\end{align*}
\end{thm}

When computing the $H_{n}\left(F(x,r)\right)$ for $r \geq 3$, they exhibit distinct patterns.
We need to discuss the cases where the $r \equiv 0 \pmod 3$, or not.

For $r \equiv 0 \pmod 3$ , we have the following Theorem and Conjectures:
\begin{thm}\label{example-F6}
\begin{align*}
H_{6n}\left(F(x,6)\right)= &H_{6n+1}\left(F(x,6)\right)= \left( -1 \right) ^{n} \left( n+1 \right) ^{5},\\
H_{6n+2}\left(F(x,6)\right)=&\left( -1 \right) ^{n}\,\frac1{10}\,\left( 4\,n+5
 \right)  \left( 2\,n+3 \right)  \left( 4\,n-3 \right)  \left( n+2
 \right)  \left( n+1 \right) ^{4} ,\\
 H_{6n+3}\left(F(x,6)\right)=&\left( -1 \right)
^{n}\frac1{45}\,
 \left( 144\,{n}^{2}+72\,n-155 \right)  \left( n+2 \right) ^{2}
 \left( 2\,n+3 \right) ^{2} \left( n+1 \right) ^{3} ,\\
H_{6n+4}\left(F(x,6)\right)=& \left( -1 \right) ^{n+1}\,\frac1{45}\,\left( n+2 \right) ^{3} \left( 2
\,n+3 \right) ^{2} \left( 144\,{n}^{2}+792\,n+925 \right)  \left( n+1
 \right) ^{2},\\
 H_{6n+5}\left(F(x,6)\right)=&\left( -1 \right) ^{n}\frac1{10}\, \left( 4\,n+15 \right)  \left( 4\,n+7 \right)
 \left( n+2 \right) ^{4} \left( 2\,n+3 \right)  \left( n+1 \right).
\end{align*}
\end{thm}

\begin{conj}
For $ r \equiv 0 \pmod 3$, we have
  \begin{align*}
   H_{rn}(F(x,r))&=H_{rn+1}(F(x,r))=(-1)^{n}(n+1)^{r-1}.\\
   H_{rn+2}(F(x,r))&+H_{rn-1}(F(x,r))= 3^2\alpha (n+1)^{r-1}.
\end{align*}
where $\lvert \alpha \rvert=\frac1{18}\,r \left( r-3 \right)$ .
\end{conj}

For $ r \equiv 1\  or\  2 \pmod 3$, the situation is more complex. We have developed software packages to compute $H_{n}(F(x,r))$ for $r \leq 27$ (this software package can also handle larger values of $r$, but there is not much necessity for it, as the method remains unchanged). 
Now we just present several results for $r = 4, 5, 7$, along with some Conjectures.

\begin{thm}\label{example-F4}
\begin{align*}
 H_{12n}(F(x,4))&=H_{12n+1}(F(x,4))=1,\\
 H_{12n+2}(F(x,4))&=2\, \left( 2\,n+1 \right)  \left( 8\,n-1 \right) ,\\
 H_{12n+3}(F(x,4))&=-2\, \left( 8\,n+9 \right)  \left( 2\,n+1
 \right) ,\\
 H_{12n+4}(F(x,4))&=H_{12n+5}(F(x,4))=-1,\\
 H_{12n+6}(F(x,4))&=4\, \left( n+1 \right)  \left( 8\,n+5 \right) ,\\
 H_{12n+7}(F(x,4))&=H_{12n+10}(F(x,4))=-64 \left( n+1 \right) ^{2},\\
 H_{12n+8}(F(x,4))&=H_{12n+9}(F(x,4))=0,\\
 H_{12n+11}(F(x,4))&=-4\,
 \left( n+1 \right)  \left( 8\,n+11 \right) .
\end{align*}
\end{thm}

\begin{thm}\label{example-F5}
\begin{align*}
 H_{15n}(F(x,5))=& H_{15n+1}(F(x,5))=1,\\
 H_{15n+2}(F(x,5))&=5\, \left( 2\,n+1 \right)  \left( 50\,{n}^{2}+15\,n-1 \right) ,\\
 H_{15n+3}(F(x,5))&=25\, \left( 25\,{n}^{2}+25\,n-2 \right)  \left( 2\,n+1 \right) ^{2},\\
 H_{15n+4}(F(x,5))&=-5\, \left( 50\,{n}^{2}+85\,n+34 \right)  \left( 2\,n+1 \right) ,\\
 H_{15n+5}(F(x,5))&=H_{15n+6}(F(x,5))=1,\\
 H_{15n+7}(F(x,5))&=5\, \left( n+1 \right)  \left( 100\,{n}^{2}+120\,n+33 \right) ,\\
 H_{15n+8}(F(x,5))&=25\, \left( n+1 \right) ^{2} \left( 100\,{n}^{2}+100\,n+17 \right) ,\\
 H_{15n+9}(F(x,5))&=H_{15n+12}(F(x,5))=1000\, \left( n+1 \right) ^{3},\\
 H_{15n+10}(F(x,5))&=H_{15n+11}(F(x,5))=0,\\
 H_{15n+13}(F(x,5))&=25\, \left( n+1 \right) ^{2} \left( 100\,{n}^{2}+300\,n+217 \right) ,\\
 H_{15n+14}(F(x,5))&=-5\, \left( n+1 \right)  \left( 100\,{n}^{2}+280\,n+193 \right) .
\end{align*}
\end{thm}

\begin{thm}\label{example-F7}
\begin{align*}
H_{21n}(F(x,7))=&H_{21n+1}(F(x,7))=\left( -1 \right) ^{n},\\
H_{21n+2}(F(x,7))=&{\left( -1 \right) ^{n+1}\, \frac {7}{30}}\, \left( 2\,n+1 \right)  \left( 374556\,{n}^{4}+385532\,{n}^{3}+127449
\,{n}^{2}+13363\,n+60 \right) ,\\
H_{21n+3}(F(x,7))=&{\left( -1
 \right) ^{n}\,\frac {49}{180}}\, ( 29647548\,{n
}^{6}+54252996\,{n}^{5}+27993259\,{n}^{4}+2344748\,{n}^{3}  \\
&-1393217\,{n
}^{2}-227934\,n-1080 )  \left( 2\,n+1 \right) ^{2},\\
 H_{21n+4}(F(x,7))=&{ \left( -1 \right) ^{n+1}\, \frac {343}{180}}\, \left( 2\,
n+1 \right) ^{3} ( 4235364\,{n}^{6}+12706092\,{n}^{5}+12161065\,{
n}^{4}\\
&+3145310\,{n}^{3}-906059\,{n}^{2}-361032\,n+3420 ) ,\\
H_{21n+5}(F(x,7))=&{ \left( -1 \right) ^{n+1}\,\frac {49}{180}}\, ( 29647548\,{n}^{6}+
123632292\,{n}^{5}+201441499\,{n}^{4}+160049288\,{n}^{3}\\
&+61715353\,{n}
^{2}+9000600\,n-123300 )  \left( 2\,n+1 \right) ^{2},\\
H_{21n+6}(F(x,7))=&{ \left( -1 \right) ^{n+1}\,\frac {7}{30}}\,\left( 374556\,{n}^{4}+1112692\,{n}^{3
}+1218189\,{n}^{2}+583163\,n+103170 \right)  \left( 2\,n+1 \right) ,\\
H_{21n+7}(F(x,7))=&H_{21n+8}(F(x,7))=\left( -1 \right) ^{n+1},\\
 H_{21n+9}(F(x,7))=&{\left( -1 \right) ^{n}\,\frac {7}{30}}\,
 \left( n+1 \right)  \left( 749112\,{n}^{4}+1850828\,{n}^{3}+1698242\,
{n}^{2}+686273\,n+103230 \right),\\
H_{21n+10}(F(x,7))=&{\left( -1 \right) ^{n+1}\,\frac {49}{
180}}\, \left( n+1 \right) ^{2} (
118590192\,{n}^{6}+375938976\,{n}^{5}+459474568\,{n}^{4}\\
&+261223312\,{n
}^{3}+59779951\,{n}^{2}-1682814\,n-1959840 ) ,\\
H_{21n+11}(F(x,7))=&{ \left( -1 \right) ^{n}\,\frac {343}{90}}
\, \left( n+1 \right) ^{3} ( 16941456\,{n}^{6}+84707280\,{n}^{5}+
148881208\,{n}^{4}+115228792\,{n}^{3}\\
&+34672057\,{n}^{2}-1047963\,n-
1697400 ),\\
H_{21n+12}(F(x,7))=&{ \left( -1 \right) ^{n}\,\frac {2401}{15}}\, \left( n+
1 \right) ^{4} \left( 115248\,{n}^{4}+460992\,{n}^{3}+587608\,{n}^{2}+
288512\,n+43695 \right),\\
H_{21n+13}(F(x,7))=&H_{21n+16}(F(x,7))=\left( -1
 \right) ^{n+1}\,537824\,\left( n+1 \right) ^{5},\\
 H_{21n+14}(F(x,7))=&H_{21n+15}(F(x,7))=0,\\
 H_{21n+17}(F(x,7))=&{ \left( -1 \right) ^{n}\,\frac {2401}{15}}\, \left( n+
1 \right) ^{4} \left( 115248\,{n}^{4}+460992\,{n}^{3}+587608\,{n}^{2}+
217952\,n-26865 \right),\\
H_{21n+18}(F(x,7))=&{ \left( -1 \right) ^{n+1}\,\frac {343}{90}}\,\left( n+1 \right) ^{3} ( 16941456\,{n}
^{6}+118590192\,{n}^{5}+318295768\,{n}^{4}\\
&+398162632\,{n}^{3}+
205815337\,{n}^{2}-2633505\,n-27024030 ) ,\\
H_{21n+19}(F(x,7))=&{ \left( -1 \right) ^{n+1}\,\frac {49}{180}}\,
\left( n+1 \right) ^{2} ( 118590192\,{n
}^{6}+1047143328\,{n}^{5}+3815496328\,{n}^{4}\\
&+7351444912\,{n}^{3}+
7906357711\,{n}^{2}+4503507834\,n+1062057240 ) ,\\
H_{21n+20}(F(x,7))=&{\left( -1 \right) ^{n+1} \,\frac {7}{30}}
\, \left( n+1 \right)  \left( 749112\,{n}^{4}
+4142068\,{n}^{3}+8571962\,{n}^{2}+7868343\,n+2702820 \right).
\end{align*}
\end{thm}

\begin{conj}
For $ r \equiv 1 \ or \ 2 \pmod 3$, we have
  \begin{align*}
   H_{3rn}(F(x,r)) &= H_{3rn+1}(F(x,r)) = H_{3rn+r}(F(x,r)) = H_{3rn+r+1}(F(x,r))=\alpha.\\
   H_{3rn+2r}(F(x,r))&=H_{3rn+2r+1}(F(x,r))= 0.\\
    H_{3rn+2r-1}(F(x,r))&+H_{3rn+2r+2}(F(x,r))=\beta \left((2r)(n+1)\right)^{r-2}.\\
     H_{3rn+2}(F(x,r))&+H_{3rn-1}(F(x,r))=\frac12\gamma r(r-3).
\end{align*}
where $\lvert \alpha \rvert=\lvert \beta \rvert=\lvert \gamma \rvert =1$.
\end{conj}

\begin{conj}
For $ r \equiv 0 \pmod 3$, the shifted periodic of $H_n(F(x,r))$ is $r$.
For $ r \equiv 1\  or\  2 \pmod 3$, the shifted periodic of $H_n(F(x,r))$ is $3r$.
\end{conj}

Our main result is the following.
\begin{thm}
  The above Conjectures \ hold for $r \leq 27$.
\end{thm}

The paper is organized as follows.
Section \ref{sec:sulanke-xin-trans} introduces Sulanke-Xin's continued fraction method, especially the quadratic transformation $\tau$.
Section \ref{calculations3-4} detailed description of the process for calculating $H_n(F(x,3))$ and $H_n(F(x,4))$ by using the Sulanke-Xin method. 
This is also the process in developing the software package.

\section{Main tool}

We will introduce the continued fraction method of Sulanke and Xin, especially their quadratic transformation $\tau$ in \cite{Sulanke-Xin}.
This is the main tool of this paper.

\subsection{Sulanke-Xin's quadratic transformation $\tau$\label{sec:sulanke-xin-trans}}
This subsection is copied from \cite{Y. Wang and G. Xin}. We include it here for the reader's convenience.

Suppose the generating function $F(x)$ is the unique solution of a quadratic functional equation which can be written as
\begin{gather}
  F(x)=\frac{x^d}{u(x)+x^kv(x)F(x)},\label{xinF(x)}
\end{gather}
where $u(x)$ and $v(x)$ are rational power series with nonzero constants, $d$ is a nonnegative integer, and $k$ is a positive integer.
We need the unique decomposition of $u(x)$ with respect to $d$: $u(x)=u_L(x)+x^{d+2}u_H(x)$ where $u_L(x)$ is a polynomial of degree at most $d+1$ and $u_H(x)$ is a power series.
Then Propositions 4.1 and 4.2 of \cite{Sulanke-Xin} can be summarized as follows.
\begin{prop}\label{xinu(0,i,ii)}
Let $F(x)$ be determined by  \eqref{xinF(x)}. Then the quadratic transformation $\tau(F)$ of $F$ defined as follows gives close connections
between   $H(F)$ and $H(\tau(F))$.
\begin{enumerate}
\item[i)] If $u(0)\neq1$, then $\tau(F)=G=u(0)F$ is determined by $G(x)=\frac{x^d}{u(0)^{-1}u(x)+x^ku(0)^{-2}v(x)G(x)}$, and $H_n(\tau(F))=u(0)^{n}H_n(F(x))$;

\item[ii)] If $u(0)=1$ and $k=1$, then $\tau(F)=x^{-1}(G(x)-G(0))$, where $G(x)$ is determined by
$$G(x)=\frac{-v(x)-xu_L(x)u_H(x)}{u_L(x)-x^{d+2}u_H(x)-x^{d+1}G(x)},$$
and we have
$$H_{n-d-1}(\tau(F))=(-1)^{\binom{d+1}{2}}H_n(F(x));$$

\item[iii)] If $u(0)=1$ and $k\geq2$, then $\tau(F)=G$, where $G(x)$ is determined by
$$G(x)=\frac{-x^{k-2}v(x)-u_L(x)u_H(x)}{u_L(x)-x^{d+2}u_H(x)-x^{d+2}G(x)},$$
and we have
$$H_{n-d-1}(\tau(F))=(-1)^{\binom{d+1}{2}}H_n(F(x)).$$\label{xinu(0)(ii)}

\end{enumerate}

\end{prop}

\section{The calculations of $H_n(F(x,3))$ and $H_n(F(x,4))$}\label{sec-HG}\label{calculations3-4}

\subsection{}
For $r=3$, we have the functional equation
\begin{align*}
F(x,3)={\frac {1}{ 1-3x+2x^3-x^6F(x,3) }}.
\end{align*}
\begin{proof}[Proof of Theorem \ref{example-F3}]
We apply Proposition \ref{xinu(0,i,ii)} to $F_0:=F(x,3)$ by repeatedly using the transformation $\tau$.
This results in a shifted periodic continued fractions of order 2:
\begin{align*}
  F_0(x)& \mathop{\longrightarrow}\limits^\tau F_1^{(1)}(x)\mathop{\longrightarrow}\limits^\tau F^{(1)}_{2}(x)\mathop{\longrightarrow}\limits^\tau F^{(1)}_{3}(x)= F_1^{(2)}(x)\cdots.
\end{align*}
We obtain
\begin{align}
  H_k(F_0)=H_{k-1}(F_1^{(1)}).\label{e-3G0}
\end{align}
For $p\geq1$, computer experiments suggest us to define
\begin{align*}
  F^{(p)}_{1} =&-{\frac {{x} \left( x^3+3p(p+1)x-p(p+1) \right)}{ p^2 x^2 F^{(p)}_{1} +2p
\,{x}^{3}+ 3p^2 x-p^2 }}.
\end{align*}
Then the results can be summarized as follows:
\begin{align*}
  H_{k-1}\left(F_{1}^{p}\right)=-\left(\frac{p+1}{p}\right)^{k-1} H_{k-3}\left(F_{2}^{p}\right) \\
H_{k-3}\left(F_{2}^{p}\right)=\left(\frac{p}{p+1}\right)^{k-3} H_{k-4}\left(F_{1}^{p+1}\right) .
\end{align*}
Combination of  the above formulas gives the recursion
\[H_{k-1}\left(F_{1}^{p}\right)=-\left(\frac{p+1}{p}\right)^{2} H_{k-4}\left(F_{1}^{p+1}\right).\]
Let $k-1=3n+j$, where $0\leq j<3$. We then deduce that
\begin{gather}
H_{3 n+j}\left(F_{1}^{1}\right)=(-1)^{n}(n+1)^{2} H_{j}\left(F_{1}^{n+1}\right).\label{e-3G3n}
\end{gather}
The initial values are
\begin{align*}
 H_{0}\left(F_{1}^{n+1}\right)=1 ; \quad H_{1}\left(F_{1}^{n+1}\right)=0 ; \quad H_{2}\left(F_{1}^{n+1}\right)=\left(\frac{n+2}{n+1}\right)^{2}.
\end{align*}
Then the theorem follows by the above initial values, \eqref{e-3G0} and \eqref{e-3G3n}.
\end{proof}


\subsection{}

\medskip
For $r=4$, we have the functional equation
\begin{align*}
  F(x,4)= -\frac1{F(x, 4) x^{8}+x^{4}-4 x^{3}-2 x^{2}+4 x-1} .
\end{align*}

\begin{proof}[Proof of Theorem \ref{example-F4}]
We apply Proposition \ref{xinu(0,i,ii)} to $F_0:=F(x,4)$ by repeatedly using the transformation $\tau$.
This results in a shifted periodic continued fractions of order 10:
\begin{align*}
  F_0(x)& \mathop{\longrightarrow}\limits^\tau F_1^{(1)}(x)\mathop{\longrightarrow}\limits^\tau F^{(1)}_{2}(x)\mathop{\longrightarrow}\limits^\tau F^{(1)}_{3}(x)\mathop{\longrightarrow}\limits^\tau F^{(1)}_{4}(x)\mathop{\longrightarrow}\limits^\tau F^{(1)}_{5}(x)\mathop{\longrightarrow}\limits^\tau F^{(1)}_{6}(x)\mathop{\longrightarrow}\limits^\tau F^{(1)}_{7}(x)\\
  &\mathop{\longrightarrow}\limits^\tau F^{(1)}_{8}(x)\mathop{\longrightarrow}\limits^\tau F^{(1)}_{9}(x)\mathop{\longrightarrow}\limits^\tau F^{(1)}_{10}(x)
  \mathop{\longrightarrow}\limits^\tau F^{(1)}_{11}(x)= F_1^{(2)}(x)\cdots.
\end{align*}
We obtain
\begin{align}
  H_k(G_0)=H_{k-1}(G_1^{(1)}).\label{e-4G0}
\end{align}
For $p\geq1$, computer experiments suggest us to define
\begin{align*}
F_{1}^{(p)}&=-\frac{1}{x^{2} F_{1}^{(p)}-x^{4}+(16 p-12) x^{3}+\left(-64 p^{2}+104 p-38\right) x^{2}+4 x-1} \times \\
&x^{6}+(-8 p+8) x^{5}+\left(96 p^{2}-148 p+52\right) x^{4}+\left(-512 p^{3}+1216 p^{2}-928 p+220\right) x^{3}+ \\
&\left(1024 p^{4}-3328 p^{3}+3920 p^{2}-1944 p+345\right) x^{2}+\left(-128 p^{2}+200 p-68\right) x+32 p^{2}-52 p+18.
\end{align*}
Then the results can be summarized as follows:
\begin{align*}
 H_{k-1}\left(F_{1}^{p}\right)&=\left(32 p^{2}-52 p+18\right)^{k-1} H_{k-2}\left(F_{2}^{p}\right) \\
H_{k-2}\left(F_{2}^{p}\right)&=\left(-\frac{8 p+1}{256 p^{3}-704 p^{2}+612 p-162}\right)^{k-2} H_{k-3}\left(F_{3}^{p}\right) \\
H_{k-3}\left(F_{3}^{p}\right)&=\left(-1 / 2 \frac{8 p-9}{(2 p-1)(8 p+1)^{2}}\right)^{k-3} H_{k-4}\left(F_{4}^{p}\right) \\
H_{k-4}\left(F_{4}^{p}\right)&=(2(8 p+1)(2 p-1))^{k-4} H_{k-5}\left(F_{5}^{p}\right) \\
H_{k-5}\left(F_{5}^{p}\right)&=(-4 p(8 p-3))^{k-5} H_{k-6}\left(F_{6}^{p}\right)\\
 H_{k-6}\left(F_{6}^{p}\right)&=\left(-4(8 p-3)^{-2}\right)^{k-6} H_{k-1}\left(F_{1}^{p}\right) H_{k-7}\left(F_{7}^{p}\right) \\
H_{k-7}\left(F_{7}^{p}\right)&=-\left(1 / 16 \frac{8 p-3}{p}\right)^{k-7} H_{k-10}\left(F_{8}^{p}\right) \\
H_{k-10}\left(F_{8}^{p}\right)&=\left(1 / 16 \frac{8 p+3}{p}\right)^{k-10} H_{k-11}\left(F_{9}^{p}\right) \\
 H_{k-12}\left(F_{9}^{p}\right)&=\left(-4(8 p+3)^{-2}\right)^{k-11} H_{k-12}\left(F_{10}^{p}\right)
\end{align*}
Combination of  the above formulas gives the recursion
\[H_{k-1}(F_1^{(p)})=H_{k-13}(F^{(p+1)}_{1}).\]
Let $k-1=12n+j$, where $0\leq j<11$. We then deduce that
\begin{gather}
 H_{12n+j+1}(F)= H_{12n+j}(F_1^{(1)})=H_{j}(F^{(n+1)}_{1}).\label{e-4G4n}
\end{gather}
The initial values are
\begin{align*}
H_{0}\left(F_{1}^{n+1}\right)&=1, \qquad \qquad\qquad\qquad\qquad\quad H_{1}\left(F_{1}^{n+1}\right)=32(n+1)^{2}-52 n-34,\\
H_{2}\left(F_{1}^{n+1}\right)&=-2\left(8 n+9\right)(2 n+1),\qquad \  \ \  \ H_{3}\left(F_{1}^{n+1}\right)=H_{4}\left(F_{1}^{n+1}\right)=-1,\\
H_{5}\left(F_{1}^{n+1}\right)&=4\left(n+1\right)\left(8n+5\right),\qquad \qquad H_{6}\left(F_{1}^{n+1}\right)=64\left(n+1\right)^2,\\
H_{7}\left(F_{1}^{n+1}\right)&=H_{8}\left(F_{1}^{n+1}\right)=0,\qquad \qquad \ \  \ 
H_{9}\left(F_{1}^{n+1}\right)=-64\left(n+1\right)^2,\\
H_{10}\left(F_{1}^{n+1}\right)&=-4\left(n+1\right)\left(8n+11\right),\qquad  \ 
H_{11}\left(F_{1}^{n+1}\right)=H_{12}\left(F_{1}^{n+1}\right)=1.
\end{align*}
Then the theorem follows by the above initial values, \eqref{e-4G0} and \eqref{e-4G4n}.
\end{proof}

\medskip

\noindent
\textbf{Declaration of Interest Statement}
\ \ The authors declare that they have no known competing financial interests or personal relationships that could have appeared to influence the work reported in this paper.\\

\medskip

\noindent
\textbf{Data Availability}\ \   Data availability is not applicable to this article as no new data were created or analyzed in this study.

\medskip

\noindent
\textbf{Acknowledgments}
\ \ The authors would like to thank zihao Zhang for his careful reading and very useful comments.\\


\end{document}